\documentclass[12pt]{amsart}

\usepackage{CJK,CJKnumb,amsmath}
\usepackage{amsfonts}
\usepackage{amsthm}
\usepackage{geometry}
\usepackage{mathrsfs}
\usepackage{amsmath}
\usepackage{amssymb}
\usepackage{amsfonts}
\usepackage[percent]{overpic}
\usepackage{bm}
\usepackage[all]{xy}
\usepackage{graphicx}
\usepackage{subfigure}
\usepackage{latexsym}
\usepackage[colorlinks, linkcolor=blue, anchorcolor=bule, citecolor=blue]{hyperref}
\usepackage{epigraph}
\usepackage{hyperref}
\usepackage{fancyhdr}
\usepackage{scalerel}
\usepackage{enumitem} 	
\usepackage{floatrow}
\usepackage{mathtools}
\mathtoolsset{showonlyrefs}

\numberwithin{equation}{section}
\usepackage{color}

\theoremstyle{remark}
\newtheorem{remark}{Remark}

\newtheorem*{ack}{Acknowledgement}
\theoremstyle{plain}
\newtheorem{thm}{Theorem}[section]
\newtheorem{thmx}{Theorem}

\newtheorem{satz}{Satz}[section]

\newtheorem{lem}[satz]{Lemma} 

\theoremstyle{definition}
\newtheorem{defn}[satz]{Definition} 

\theoremstyle{remark}

\numberwithin{equation}{section}

\newcommand{\norm}[1]{\left\lVert#1\right\rVert}

\newcommand{\La}{\Lambda}
\newcommand{\C}{\mathbb{C}}				
\newcommand{\R}{\mathbb{R}}				
\newcommand{\N}{\mathbb{N}}				
\newcommand{\Z}{\mathbb{Z}}				

\newcommand{\ul}{\underline}					

\newcommand{\normz}[1]{\left\lVert#1\right\rVert_2}

\newcommand{\abs}[1]{\left\lvert#1 \right\rvert}

\newcommand{\gdw}{  \relax  \ifmmode \Longleftrightarrow  \else    $\Longleftrightarrow$  \fi}

\DeclareMathOperator*{\esssup}{ess\,sup}
\DeclareMathOperator*{\Comp}{\scalerel*{\bigcirc}{\bigcirc}}

\begin{document}
\title{On the differentiability of hairs for Zorich maps}
\author{Patrick Comd{\"u}hr}

\address{Mathematisches Seminar, Christian-Albrechts-Universit\"at zu Kiel, Ludewig-Meyn-Str. 4, D-24098 Kiel, Germany.}
\email{comduehr@math.uni-kiel.de}


\begin{abstract}
Devaney and Krych showed that for the exponential family $\lambda e^z$, where $0<\lambda <1/e$, the Julia set consists of uncountably many pairwise disjoint simple curves tending to $\infty$. Viana proved that these curves are smooth. In this article we consider a quasiregular counterpart of the exponential map, the so-called Zorich maps, and generalize Viana's result to these maps.
\end{abstract}

\maketitle
\section{Introduction and main result} \label{Intro}
 For an entire function $f$ the Julia set $J(f)$ of $f$ is the set of all points in $\C$ where the iterates $f^k$ of $f$ do not form a normal family in the sense of Montel. Given an attracting fixed point $\xi$ of $f$ we denote by $A(\xi):=\{z : \lim_{k\to\infty} f^k(z)=\xi \}$ the \textit{basin of attraction} of $\xi$. From the theory of complex dynamics it is well-known that $J(f)=\partial A(\xi)$, see \cite[Corollary 4.12]{Mi06}. For further information on complex dynamics we refer to \cite{Bea91,Ber93,Mi06,St93}.

Devaney and Krych \cite{DK84} showed that for $f(z)=\lambda e^z$, where $0<\lambda<1/e$, there exists an attracting fixed point $\xi \in \R$ such that $J(f)=\C\setminus A(\xi)$ and gave a detailed description of the structure of $J(f)$. We say that a subset $H$ of $\C$ (or $\R^d$) is a \textit{hair}, if there exists a homeomorphism $\gamma \colon [0,\infty) \to H$ such that $\lim_{t\to\infty} \gamma(t)=\infty$. We call $\gamma(0)$ the \textit{endpoint} of the hair.

We only state the part of the result due to Devaney and Krych which is relevant for us.
\begin{thmx} \label{ThmA}
	For $0<\lambda <1/e$ the set $J(\lambda e^z)$ is an uncountable union of pairwise disjoint hairs.
\end{thmx}
For a set $X$ in $\C$ (or in $\R^d$) we denote by $\operatorname{dim} X$ the \textit{Hausdorff dimension} of $X$. The following result is due to McMullen \cite{McM87}.

\begin{thmx} \label{ThmB}
	Let $\lambda\in \C\setminus\{0\}$. Then $\operatorname{dim} J(\lambda e^z)=2$.
\end{thmx}

McMullen's result implies that in the situation of Theorem \ref{ThmA} the union of the hairs has Hausdorff dimension 2. The following result of Karpi\'nska \cite{Ka99} is also known as Karpi\'nska's paradox, see e.g. \cite{SZ03}.
\begin{thmx} \label{ThmC}
	Let $0<\lambda < 1/e$ and let $\mathcal{C}$ be the set of endpoints of the hairs that form $J(\lambda e^z)$. Then $\operatorname{dim} \mathcal{C} = 2$ and $\operatorname{dim} (J(\lambda e^z)\setminus \mathcal{C}) = 1$.
\end{thmx}

The existence of hairs is not restricted to the situation considered by Devaney and Krych. Hairs appear for $\lambda e^z$ for all $\lambda \in \C\setminus\{0\}$ (see \cite{DGH86,SZ03}) and also for more general classes of functions (see \cite{Ba07,DT86,RRRS11}). 

For further information on dynamics of exponential functions we refer to papers by Rempe \cite{Re06} and Schleicher \cite{Sch03}.

Viana \cite{Vi88} investigated the differentiability of hairs for exponential maps. 

\begin{thmx} \label{ThmD}
For all $\lambda\in \C\setminus\{0\}$ the hairs of $\lambda e^z$ are $C^\infty$-smooth.
\end{thmx}

In this paper we consider a higher-dimensional analog of exponential maps, the so-called Zorich maps. These maps are quasiregular maps, which can be considered as a higher-dimensional analog of holomorphic maps. Since we will not use any results about quasiregular maps, we do not give the definition here, but refer to Rickman's monograph \cite{Ri93}. We note, however, that the quasiregularity is an underlying idea in many of the arguments.

 Every Zorich map depends on a bi-Lipschitz map which maps a square to some body, for example to a hemisphere or to the faces of a tetrahedron (see \cite[page 121]{IM01}). Recently, Nicks and Sixsmith \cite{NS17} considered a bi-Lipschitz map which maps a square to the faces of a square based pyramid to obtain a quasiregular map $f\colon \R^3 \to\R^3$ of transcendental type, which has a periodic domain where all iterates of $f$ tend locally uniformly to $\infty$. In this paper we restrict ourself to the `standard` case, i.e. our bi-Lipschitz maps always map a square to the upper or lower hemisphere. 
 
 For the definition of these maps we follow \cite[page 119]{IM01}. Subsequently we summarize some results of Bergweiler \cite{Ber10}, but we replace the dimension 3 by any dimension $d\geq 3$. 
 
We define for $d\in\N$ with $d\geq 3$ the hypercube
\begin{equation}
Q:=\{x\in \R^{d-1} : \norm{x}_{\infty}\leq 1\} = [-1,1]^{d-1}, 
\end{equation}
the upper hemisphere
\begin{equation}
\mathbb{S}_+ := \{x\in\R^d : \norm{x}_2 = 1 \text{ and } x_d\geq 0\} 
\end{equation}
and for $c\in \R$ the half-space
\begin{equation}
\mathbb{H}_{\geq c} := \{x\in \R^d : x_d \geq c\}.
\end{equation}
The half-spaces $\mathbb{H}_{>c}$, $\mathbb{H}_{<c}$ and $\mathbb{H}_{\leq c}$ are defined analogously.
For a bi-Lipschitz map $h\colon Q \to \mathbb{S}_+$ we define 
\begin{equation}
F\colon Q\times \R \to \mathbb{H}_{\geq 0},  \, F(x)=e^{x_d}\,h(x_1,\dots ,x_{d-1}).
\end{equation}
By reflection we get a function $F \colon \R^d \to \R^d$ which we call a \textit{Zorich map}.

If $DF(x_1,\dots,x_{d-1},0)$ exists, we have 
\begin{equation}
DF(x_1,\dots,x_{d-1},x_d)=e^{x_d}\,DF(x_1,\dots,x_{d-1},0) \label{eq:DF(0)}
\end{equation} 
and thus there exist $\alpha ,m, M \in \R$, $\alpha \in (0,1)$, $m<M$, $M\geq 1$ such that 
\begin{align}
\norm{DF(x)} &\leq \alpha \quad \, \text{    a.e. for } x_d\leq m \label{eq:alpha}  \intertext{and} 
l(DF(x))&\geq \frac{1}{\alpha} \quad \text{   a.e. for } x_d \geq M, \label{eq:lDF}
\end{align}
where $DF(x)$ denotes the derivative,
\begin{equation}
\norm{DF(x)}=\sup\limits_{\normz{h}=1}\normz{DF(x)h}
\end{equation}
the operator norm of $DF(x)$ and 
\begin{equation}
l(DF(x))=\inf\limits_{\normz{h}=1}\normz{DF(x)h}.
\end{equation}
Consider now for $a\in \R$ such that 
\begin{equation}
a\geq e^M - m \label{eq:a} %
\end{equation}
the map
\begin{equation}
f_a \colon \R^d \to \R^d,\, f_a(x)=F(x)-(0,\dots,0,a). 
\end{equation}
In our context we call $f_a$ a Zorich map, too. 
The following theorem is due to Bergweiler \cite{Ber10}. 
\begin{thmx} \label{bergweiler1}
	Let $f_a$ be as above with $a$ as in \eqref{eq:a}. Then there exists a unique fixed point $\xi=(\xi_1,\dots,\xi_{d})$ satisfying $\xi_d \leq m$, the set
	\begin{equation}
	J:=\{x\in \R^d \colon f_a^k(x) \nrightarrow \xi\}    
	\end{equation}
	consists of uncountably many pairwise disjoint hairs, and the set $\mathcal{C}$ 
	of endpoints of these hairs has Hausdorff dimension $d$, while $J\setminus \mathcal{C}$ has Hausdorff dimension 1.
\end{thmx}
Thus Theorem \ref{bergweiler1} corresponds to Theorem \ref{ThmA}, \ref{ThmB} and \ref{ThmC}. 
For a set $\Omega \subset \R^{d-1}$ we denote by $\operatorname{int}(\Omega)$ the \textit{interior} of $\Omega$. The main result of this article is the following.
\begin{thm} \label{THM}
	Let $f_a$ be as above with $a$ as in \eqref{eq:a} and assume that $h|_{\operatorname{int}(Q)}$ is $C^1$ and $Dh$ is H\"older continuous. Then the hairs of $f_a$ are $C^1$-smooth.
\end{thm}
	This theorem corresponds to Theorem \ref{ThmD} concerning the differentiability of hairs. As in Viana's result, we do not obtain differentiability of the hairs in the endpoints since they can spiral around a point. Rempe \cite[3.4.2 Theorem, page 31]{Re03} gave a condition under which we obtain smoothness up to the endpoints for exponential maps. The proof of Theorem \ref{THM} will show that the conclusion of the theorem still holds if the assumptions on $h$ and $Dh$ are satisfied on a suitable subset of $\operatorname{int}(Q)$ (see Section \ref{Remarks}).

\begin{ack}
	I would like to thank my supervisor Walter Bergweiler, Dave Sixsmith and the referee for valuable suggestions.
\end{ack}

\section{Preliminaries}
In this section we are recalling mainly results from \cite{Ber10}, formulating them for functions in $\R^d$ with $d\geq 3$ instead of $\R^3$, however.
For simplicity we write $f=(f_1,\dots,f_d)$ instead of $f_a$. For $r\in \Z^{d-1}$ we denote by
\begin{align}
P(r):=P(r_1,\dots,r_{d-1}) &:=\{x\in\R^{d-1}\colon \forall j\in\{1,\dots,d-1\} \colon \abs{x_j-2r_j}<1\}\\
&\hspace{0.1cm}= \operatorname{int}(Q) +(2r_1,\dots,2r_{d-1}) 
\end{align}
the shifted and open square $Q$ with centre $2r$.  We put
\begin{equation}
S:=\left\{r\in \Z^{d-1}\colon \sum_{j=1}^{d-1} r_j \in 2\Z  \right\}.
\end{equation}
Then 
\begin{equation}
F \text{ maps } P(r)\times \R \text{ bijectively onto }
\begin{dcases}
\mathbb{H}_{>0}, & \text{if } r \in S \\
\mathbb{H}_{<0}, & \text{if } r \notin S \\
\end{dcases}
\end{equation}
and thus
\begin{equation} \label{f maps}
f \text{ maps } P(r)\times \R \text{ bijectively onto }
\begin{dcases}
\mathbb{H}_{>-a}, & \text{if } r \in S, \\
\mathbb{H}_{<-a}, & \text{if } r \notin S. \\
\end{dcases}
\end{equation}
\begin{defn}[Tract]
	For $r\in S$ we call the set
	\begin{equation}
	T(r):=P(r)\times (M,\infty) 
	\end{equation}
	the \textit{tract} above $P(r)$.
\end{defn}
Now we want to understand the behaviour of our function $f$ and collect some important facts about it. Since $f(P(r)\times \R) = \mathbb{H}_{>-a}$ for $r\in S$ and
\begin{equation}
f_d(x_1,\dots,x_d) = e^{x_d}h_d(x_1,\dots,x_{d-1})-a \leq e^{x_d} -e^M +m \leq m < M \label{eq:f_d<M}
\end{equation}
for $x_d\leq M$ and hence $f(P(r)\times (-\infty,M])\subset \mathbb{H}_{<M}$, we have $f(T(r))\supset \mathbb{H}_{\geq M}$. Thus there exists a branch $\Lambda^r \colon \mathbb{H}_{\geq M} \to T(r)$ of the inverse function of $f$. Using the notation $\Lambda:=\Lambda^{(0,\dots,0)}$, we have
\begin{equation}
\Lambda^{(r_1,\dots,r_{d-1})}(x) = \Lambda(x)+(2r_1,\dots,2r_{d-1},0) \label{eq:lambda_shift}  
\end{equation}
for all $x\in \mathbb{H}_{\geq M}$ and $r\in S$. 

In our proofs we will often use the derivative of $\Lambda^r$. Together with \eqref{eq:lambda_shift} we obtain
\begin{equation}
D\Lambda^r(x) = D\Lambda(x)  \label{eq:deriv}
\end{equation}
for all $x\in \mathbb{H}_{\geq M}$ for which the derivative exists. Because $DF=Df$, we deduce from \eqref{eq:alpha} that 
\begin{equation}
\norm{D\Lambda(x)}\leq \alpha \quad \text{a.e.} \label{eq:deriv<alpha}
\end{equation}
for $x\in\mathbb{H}_{\geq M}$. This implies for $x,y\in \mathbb{H}_{\geq M}$ that
\begin{equation}
\normz{\La(x)-\La(y)} \leq \normz{x-y}\esssup\limits_{z\in[x,y]}\norm{D\La(z)} \leq \alpha \normz{x-y}. \label{eq:L(x)-L(y)}
\end{equation}
Since \eqref{eq:f_d<M} and thus $f(\mathbb{H}_{\leq M})\subset \mathbb{H}_{\leq m}$, we obtain using \eqref{eq:lDF}
\begin{equation}
\normz{f(x)-f(y)}\leq \normz{x-y}\esssup\limits_{z\in[x,y]}\norm{Df(z)} \leq \alpha \normz{x-y}.
\end{equation}
So Banach's fixed point theorem gives us the existence of a unique fixed point $\xi\in\mathbb{H}_{\leq m}$ such that $\lim\limits_{n\to\infty} f^n(x)=\xi$ for all $x\in\mathbb{H}_{\leq m}$. Together with \protect \eqref{f maps} this yields
\begin{equation}
J \subset \bigcup_{r\in S} T(r).
\end{equation}
In the following we collect some estimates for $D\La$. The equation \protect\eqref{eq:DF(0)} implies that there exist $c_1, c_2 >0$ such that
\begin{equation}
c_1e^{x_d} \leq l(Df(x)) \leq \norm{Df(x)} \leq c_2e^{x_d} \quad \text{a.e.}
\end{equation}
Thus there are constants $c_3>0$ and $c_4\geq 1$ such that for $x\in\mathbb{H}_{\geq M}$
\begin{align}
l(D\La(x)) &\geq \frac{c_3}{\normz{x}}\quad \text{a.e.} \\
\norm{D\Lambda (x)} &\leq \frac{c_4}{\normz{x}} \quad \text{a.e.} \label{eq:inverse} 
\end{align}
Moreover, there exist $c_5>0$ and $c_6\geq 1$ such that
\begin{equation}
\frac{c_5}{\normz{x}^d} \leq  J_\La (x) \leq \frac{c_6}{\normz{x}^d} \quad \text{a.e.},
\end{equation}
where $J_\La(x)$ denotes the Jacobian determinant.

Let us fix $x,y\in\mathbb{H}_{\geq M}$. Then we can connect $x$ and $y$ by a path $\gamma$ in
\begin{equation}
\mathbb{H}_{\geq M} \cap \Big\{z\in\R^d \colon \normz{z}\geq \operatorname{min}\{\normz{x},\normz{y}\}\Big\}
\end{equation}
with $\operatorname{length(\gamma)}\leq \pi\normz{x-y}$. Together with \eqref{eq:inverse} this yields
\begin{equation}
\normz{\La(x)-\La(y)} \leq \pi \normz{x-y}\esssup\limits_{z\in\gamma}\norm{D\La(z)}\leq c_4\pi \dfrac{\normz{x-y}}{\operatorname{min}\{\normz{x},\normz{y}\}}. \label{eq:c_4/min}
\end{equation}
Since $h$ is bijective, there exists a unique point $(v_1,\dots,v_{d-1})\in Q$ which is mapped under $h$ to the north pole of the sphere, i.e.
\begin{equation}
h(v_1,\dots,v_{d-1})=(0,\dots,0,1). \label{eq:v}
\end{equation}
 This implies $F(v_1,\dots,v_{d-1},x_d)=(0,\dots,0,e^{x_d})$ and for $r\in S$ and $t\geq M$ this yields
\begin{equation}
\La^r(0,\dots,0,t) = \big(v_1+2r_1,\dots,v_{d-1}+2r_{d-1},\log(t+a)\big).
\end{equation}
Moreover, the equation
\begin{equation}
\normz{F(x)}=e^{x_d} 
\end{equation}
yields
\begin{equation}
\normz{f(\La(x))+(0,\dots,0,a)}=e^{\La_d(x)}
\end{equation}
and thus
\begin{equation}
\La_d(x) = \log(\normz{x+(0,\dots,0,a)}).   \label{eq:La_d}
\end{equation}
To discuss the existence of hairs, it is useful, following Schleicher and Zimmer \cite{SZ03}, to define a reference function 
\begin{equation}
E\colon [0,\infty) \to [0,\infty),\, E(t)=e^t -1.  
\end{equation}
This function has the following properties: \\

\noindent 1. We have $E(0)=0$ and $\lim\limits_{k\to \infty}E^k(t)=\infty$ for $t>0$. \\
2. For $b>1$ we have
\begin{equation}
E^k(t)<\log\left(E^{k+1}(t)+b\right)\leq E^k(t)+\log(b). \label{eq:log Ek+1 = Ek+Rk}
\end{equation}
3. If $0<t'<t''<\infty$, then 
\begin{equation}
\lim\limits_{k\to \infty}\big(E^k(t'')-E^k(t')\big)=\infty \quad \text{and} \quad \lim\limits_{k\to \infty}\dfrac{E^k(t'')}{E^k(t')}=\infty. \label{eq:E(t)}
\end{equation}
We need an analogous definition as in the case of exponential maps:
\begin{defn}[External address/Symbolic space]
	For each $x\in J$ we call the sequence
	\begin{equation}
	\ul{s}(x):=s_0 s_1 s_2 \dots = (s_k)_{k\geq 0} \in S^{\N_0}   
	\end{equation}
	such that $f^k(x)\in T(s_k)$ for all $k\geq 0$ the \textit{external address of $x$}. Moreover, we call $\Sigma:=S^{\N_0}$ the \textit{symbolic space}.  
\end{defn}
\begin{defn}[Admissibility/exponentially boundedness]
	We say that $\ul{s}\in \Sigma$ is \textit{admissible} (or \textit{exponentially bounded}), if there exists a $t>0$ such that
	\begin{equation}
	\limsup\limits_{k\to\infty} \frac{\normz{s_k}}{E^k(t)}< \infty.
	\end{equation}
	Moreover, we denote by $\Sigma' \subset \Sigma$ the set of all admissible points. 
\end{defn}
With these definitions we obtain the following lemmas, see \cite[Propositions 3.1 and 3.2]{Ber10}.
\begin{lem} \label{prop1}
	Let $x\in J$. Then $\ul{s}(x)$ is admissible.
\end{lem}
\begin{lem} \label{prop2}
	Let $\ul{s}\in \Sigma'$. Then $\{x\in J : \ul{s}(x)=\ul{s}\}$ is a hair.
\end{lem}
From Lemma \ref{prop1} and Lemma \ref{prop2} it follows that $J$ is the union of hairs as stated in Theorem \ref{bergweiler1}.

Fixing $\ul{s} \in \Sigma'$, we denote
\begin{equation}
t_{\ul{s}} := \inf \bigg\{t>0\colon  \limsup\limits_{k\to\infty} \frac{\normz{s_k}}{E^k(t)}< \infty \bigg\}.
\end{equation}
Choosing $t_k\in [0,\infty)$ such that $2\normz{s_k}=E^k(t_k)$ and putting $\tau_k:=\sup\limits_{j\geq k} t_j$, we obtain
\begin{equation}
t_{\ul{s}}=  \limsup\limits_{k\to\infty} t_k =  \lim\limits_{k\to\infty} \tau_k. \label{eq:t_s}
\end{equation}
Using the abbreviation
\begin{equation}
L_k := \La^{s_k} =\La^{(s_{k,1},\dots,s_{k,d-1})}
\end{equation}
we define for $k\in\N_0$
\begin{equation}
g_k\colon [0,\infty)\to \mathbb{H}_{\geq M},\, g_k(t)=(L_0\circ L_1 \circ \dots \circ L_k)\left(0,\dots,0,E^{k+1}(t)+M\right). \label{eq:g_k}
\end{equation}
The two main lemmas in the proof of Lemma \ref{prop2} are the following \cite[Lemmas 3.1 and 3.2]{Ber10}.
\begin{lem}
	The sequence $(g_k)_{k\geq 0}$ converges locally uniformly on $(t_{\underline{s}},\infty)$.
\end{lem}
\begin{lem}
	The sequence $(g_k)_{k\geq 0}$ has a subsequence which converges uniformly on $[t_{\ul{s}},\infty)$ and thus $g$ extends to a continuous map $g\colon [t_{\ul{s}},\infty) \to \mathbb{H}_{\geq M}$.
\end{lem}

\section{Proof of Theorem \ref{THM}}

To have a chance for a $C^1$ condition for our hairs, we need enough regularity of the bi-Lipschitz mapping $h$. In this section we want to specify this condition and want to give precise $C^1$ estimates for the hairs. 

Therefore we will introduce some new notations. For $n,m\in\N$ with $m<n$ and functions $f_1,\dots, f_n \colon \R^d \to \R^d$ we denote
\begin{equation}
\Comp_{j=m}^{n}f_j := f_m \circ \dots \circ f_n,
\end{equation}
where we use the convention
\begin{equation}
\Comp_{j=m+1}^{m}f_j := id. 
\end{equation}
Then $g_k$ defined by \protect \eqref{eq:g_k} takes the form
\begin{equation}
g_k(t)= \left(\Comp_{j=0}^{k}L_j\right)\left(0,\dots,0,E^{k+1}(t)+M\right) 
\end{equation}
for all $k\in\N_0$ and $t\in [0,\infty)$. 
If $h$ is locally $C^1$, the derivative $g'_k$ then reads as
\begin{equation}
\begin{aligned}
g'_k(t)&= \dfrac{d}{dt}\left(\left(L_0\circ \dots \circ L_{k}\right)\left(0,\dots,0,E^{k+1}(t)+M\right)\right) \\
&=\prod_{l=1}^{k+1}\left(D\La\left(\left(\Comp_{j=l}^{k}L_j\right)(0,\dots,0,E^{k+1}(t)+M) \right)  \right) \\
&\hspace{1cm}\cdot \left(0,\dots,0,(E^{k+1})'(t)\right)^T \label{eq:g_k'}
\end{aligned}
\end{equation}
for $k\in\N$. 

To prove Theorem \ref{THM}, it is enough to show the following result.
\begin{thm} \label{thm1}
	Let $f$ be as before and assume that $h|_{\operatorname{int}(Q)}$ is $C^1$ and $Dh$ is H\"older continuous. Then for all $\ul{s}\in \Sigma'$ the sequence $(g_k)_{k\geq 0}$ consists of $C^1$-curves which converge (in $C^1$-sense) locally uniformly on $(t_{\ul{s}},\infty)$.
\end{thm}
For the proof of the theorem we will compare $g_k'$ and $g_{k-1}'$ in a suitable way. Therefore we define as a preparation for all $k\in\N_0$ the auxilary function
\begin{equation}
\phi_k \colon [0,\infty)\to \mathbb{H}_{\geq M},\, \phi_k(t)=L_k(0,\dots,0,E^{k+1}(t)+M). \label{eq:phi_k}
\end{equation}
Then we obtain for all $t\geq 0$
\begin{equation}
\phi_k(t) = \big(v_1+2s_{k,1},\dots,v_{d-1}+2s_{k,d-1},\log(E^{k+1}(t)+M+a)\big).
\end{equation}
Thus
\begin{equation}
\phi'_k(t)=\left(0,\dots,0,\frac{1}{E^{k+1}(t)+M+a}E'(E^k(t))\cdot (E^k)'(t)\right) \label{eq:phi1}
\end{equation}
and
\begin{equation}
\phi_k'(t) = D\La\left(0,\dots,0,E^{k+1}(t)+M\right) \cdot \left(0,\dots,0,\left(E^{k+1}\right)'(t)\right)^T. \label{eq:phi2}
\end{equation}
Noticing that 
\begin{equation}
a\geq e^M - m \geq 1+M-m > 1, \label{eq:a>1}
\end{equation}
we obtain
\begin{equation}
\normz{\phi_k'(t)} = \dfrac{E^{k+1}(t)+1}{E^{k+1}(t)+M+a}\cdot (E^k)'(t) \leq (E^k)'(t).  \label{eq:phi_k' < E^k'}
\end{equation}

\begin{lem} \label{Rschlange}
	For all $k\in \N_{\geq 2}$ and $l\in \{1,\dots,k-1\}$ and for all $t \geq 0$ we have 
	\begin{equation}
	\normz{\left(\Comp_{j=l}^{k-1}L_j\right)(\phi_k(t))}\geq \left(\Comp_{j=l}^{k-1}L_j\right)_d(\phi_k(t))  \geq E^{l}(t). \label{eq:E^l(t)}
	\end{equation}
\end{lem}
\begin{proof}
	From \protect \eqref{eq:La_d}, \protect \eqref{eq:log Ek+1 = Ek+Rk} and \protect \eqref{eq:a>1} we deduce that
	\begin{align}
	\normz{L_{k-1}(\phi_k(t))} &\geq \abs{L_{k-1,d}(\phi_k(t))} \\
	&= \abs{\La_d(\phi_k(t))} \\
	&= \log(\normz{\phi_k(t)+(0,\dots,0,a)})\\
	&\geq \log\left(\phi_{k,d}(t)+a\right) \\
	&= \log\left(\log\left(E^{k+1}(t)+M+a\right)+a\right) \\
	&\geq \log(E^k(t)+a) \\
	&\geq E^{k-1}(t).
	\end{align}
	Take now $l<k-1$ such that \eqref{eq:E^l(t)} is true with $l$ replaced by $l+1$. Then we obtain
	\begin{align}
	\normz{\left(\Comp_{j=l}^{k-1}L_j\right)(\phi_k(t))} &= \normz{L_l\circ \left(\Comp_{j=l+1}^{k-1}L_j\right)(\phi_k(t))} \\
	&\geq \abs{L_{l,d}\left(\left(\Comp_{j=l+1}^{k-1}L_j\right)(\phi_k(t))\right)} \\
	&= \log\left(\normz{\left(\Comp_{j=l+1}^{k-1}L_j\right)(\phi_k(t))+(0,\dots,0,a)}\right) \\
	&\geq \log\left(\left(\Comp_{j=l+1}^{k-1}L_j\right)_d(\phi_k(t))+a\right) \\
	&\geq \log(E^{l+1}(t)+a) \\
	&\geq E^l(t)
	\end{align}
	which proves the result.
\end{proof}
\begin{remark} \label{doubleschlange}
	The argument shows that the conclusion of Lemma \ref{Rschlange} also holds if $\phi_k(t)$ is replaced by $(0,\dots,0,E^k(t)+M)$.
\end{remark}
Since the operatornorm of $DF(x)$ is comparable to the maximum of all entries of this matrix, there exists a constant $C>0$ such that 
\begin{equation}
\norm{DF(x)}\leq C \max_{\substack{1\leq j\leq d \\ 1\leq k \leq d}}  \abs{DF_{jk}(x)} \label{eq:oper_max_entry}.
\end{equation}
In the following let $\beta\in (0,1]$ and let $Dh$ be $\beta$-H\"older continuous, i.e. there is a constant $H_\beta>0$ such that
\begin{equation}
\norm{Dh(x)-Dh(y)} \leq H_\beta \normz{x-y}^\beta
\end{equation}
for all $x,y\in \mathbb{H}_{\geq M} $. Moreover, we denote by $L_h$ the Lipschitz constant of $h$. 

We want to use this to prove the following Lemma:
\begin{lem} \label{estimate_I_2}
	If $h$ is as in Theorem \ref{thm1}, then there is a constant $c_7>0$ such that for all $x,y\in \mathbb{H}_{\geq M}$ 
	\begin{equation}
	\norm{DF(\La(x))-DF(\La(y))} \leq c_7\min\left\{\normz{x}^{1-\beta},\normz{y}^{1-\beta}\right\}  \cdot \max\left\{\normz{x-y},\normz{x-y}^\beta\right\}. 
	\end{equation}
\end{lem}
\begin{proof}
	We have for $x=(\tilde{x},x_d),y=(\tilde{y},y_d)\in \mathbb{H}_{\geq M}$ using \protect\eqref{eq:oper_max_entry} 
	\begin{align}
	\norm{DF(x)-DF(y)}  &\leq C \max_{\substack{1\leq j\leq d \\ 1\leq k \leq d}}  \abs{DF_{jk}(x)-DF_{jk}(y)} \\
	&\leq C \max_{\substack{1\leq j\leq d \\ 1\leq k \leq d}}  \abs{e^{x_d}DF_{jk}(\tilde{x},0)-e^{y_d}DF_{jk}(\tilde{y},0)}.
	\end{align}
	Because $\abs{DF_{jk}(x)-DF_{jk}(y)}$ is symmetric in $x$ and $y$, we assume without loss of generality that $y_d \leq x_d$. Then we obtain
	\begin{align}
	&\hspace{0.49cm}\abs{e^{x_d}DF_{jk}(\tilde{x},0)-e^{y_d}DF_{jk}(\tilde{y},0)} \\
	&\leq e^{y_d}\abs{DF_{jk}(\tilde{x},0)-DF_{jk}(\tilde{y},0)}+(e^{x_d}-e^{y_d})\cdot\max\{\abs{DF_{jk}(\tilde{x},0)},\abs{DF_{jk}(\tilde{y},0)}\} .
	\end{align}
	Since $h$ is as in Theorem \ref{thm1} and for all $z\in\R^d$ and $j,k \in \{1,\dots,d\}$
	\begin{equation}
	DF_{jk}(\tilde{z},0) =
	\begin{dcases}
	\frac{\partial}{\partial x_k} h_j(\tilde{z}), & \text{if } k\leq d-1, \\
	\quad h_j(\tilde{z}), & \text{if } k=d, \\
	\end{dcases}
	\end{equation}
	we have with $\tilde{C}=\max\{L_h,H_\beta\}\geq 1$, noting that $\abs{\frac{\partial}{\partial x_k} h_j(\tilde{z})}\leq L_h$,
	\begin{align}
	\norm{DF(x)-DF(y)} &\leq C \left(\tilde{C}e^{y_d}\cdot\max\left\{\normz{\tilde{x}-\tilde{y}},\normz{\tilde{x}-\tilde{y}}^\beta\right\}+L_h\cdot (e^{x_d}-e^{y_d}) \right) \\
	&\leq C\tilde{C} \left(e^{y_d}\cdot\max\left\{\normz{x-y},\normz{x-y}^\beta\right\}+(e^{x_d}-e^{y_d}) \right).
	\end{align}
	Using the fact that $\min\{\normz{x},\normz{y}\}\geq M\geq 1$ and \eqref{eq:c_4/min} we obtain
	\begin{equation}
	\begin{aligned}
	\tilde{M}&:=\max\left\{\normz{\La(x)-\La(y)},\normz{\La(x)-\La(y)}^\beta\right\}  \\
	&\leq \max\left\{c_4\pi\dfrac{\normz{x-y}}{\min\{\normz{x},\normz{y}\}},c_4^\beta\pi^\beta\dfrac{\normz{x-y}^\beta}{\min\{\normz{x}^\beta,\normz{y}^\beta\}} \right\} \\
	&\leq c_4\pi\dfrac{\max\left\{\normz{x-y},\normz{x-y}^\beta\right\}}{\min\left\{\normz{x}^\beta,\normz{y}^\beta\right\}}.
	\end{aligned}
	\end{equation}
	This yields together with \protect\eqref{eq:c_4/min}, \protect \eqref{eq:La_d}, $\normz{x+(0,\dots,0,a)}\leq \normz{x}+a$ and replacing $x$ and $y$ by $\La(x)$ and $\La(y)$
	\begin{equation}
	\begin{aligned}
	&\hspace{0.49cm}\norm{DF(\La(x))-DF(\La(y))}  \\
	&\leq C\tilde{C} \left(\left(\min\{\normz{x},\normz{y}\}+a\right)\cdot\tilde{M}+\abs{\normz{x}-\normz{y}}   \right) \\
	&\leq C\tilde{C} \left(c_4\pi \dfrac{\min\left\{\normz{x},\normz{y}\right\}+a}{\min\left\{\normz{x}^\beta,\normz{y}^\beta\right\}}+1\right) \cdot  \max\left\{\normz{x-y},\normz{x-y}^\beta\right\} \\
	&\leq C\tilde{C}c_4\pi \left(\dfrac{\min\left\{\normz{x},\normz{y}\right\}+a}{\min\left\{\normz{x}^\beta,\normz{y}^\beta\right\}}+1\right) \cdot  \max\left\{\normz{x-y},\normz{x-y}^\beta\right\}.  
	\end{aligned}
	\end{equation}
	Since for $x,y\in \mathbb{H}_{\geq M}$ we have $\min\{\normz{x},\normz{y}\}\geq \min\left\{\normz{x}^\beta,\normz{y}^\beta\right\}\geq 1$ and
	\begin{align}
	\dfrac{\min\left\{\normz{x},\normz{y}\right\}+a}{\min\left\{\normz{x}^\beta,\normz{y}^\beta\right\}}+1 &= \dfrac{\min\left\{\normz{x},\normz{y}\right\}+a+\min\left\{\normz{x}^\beta,\normz{y}^\beta\right\}}{\min\left\{\normz{x}^\beta,\normz{y}^\beta\right\}} \\
	&\leq (2+a)\cdot \min\left\{\normz{x}^{1-\beta},\normz{y}^{1-\beta}\right\},
	\end{align}
	there exists $c_7>0$ such that we have for all $x,y\in \mathbb{H}_{\geq M}$
	\begin{equation}
	\begin{aligned}
	&\hspace{0.49cm}\norm{DF(\La(x))-DF(\La(y))} \\
	&\leq c_7\min\left\{\normz{x}^{1-\beta},\normz{y}^{1-\beta}\right\}  \cdot \max\left\{\normz{x-y},\normz{x-y}^\beta\right\}, \label{eq:DF(Lx)-DF(Ly)}
	\end{aligned}
	\end{equation}
	which proves the Lemma.
\end{proof}
\begin{lem}\label{estimate_1}
	For all $k\in\N$ and $t\in (0,\infty)$ we have 
	\begin{equation}
	\normz{\dfrac{d}{dt}(L_{k-1}\circ \phi_k)(t)-\dfrac{d}{dt}(L_{k-1}(0,\dots,0,E^k(t)+M))} \leq \big(I_{1,k}(t)+I_{2,k}(t)\big) \cdot (E^k)'(t)
	\end{equation}
	with
	\begin{equation}
	I_{1,k}(t) := \frac{ac_4}{E^{k+1}(t)\cdot E^k(t)}
	\end{equation}
	and
	\begin{equation}
	I_{2,k}(t) := c_8 \cdot \dfrac{2\normz{s_k}+1}{E^k(t)^{1+\beta}}, \label{eq:I_2,k}
	\end{equation}
	where $c_8:= c_4^2c_7(d+\log(M+a)+M)$.
\end{lem}
\begin{proof}
	By the triangle inequality we have
	\begin{equation}
	\begin{aligned}
	&\hspace{0.49cm}\normz{\dfrac{d}{dt}(L_{k-1}\circ \phi_k)(t)-\dfrac{d}{dt}(L_{k-1}(0,\dots,0,E^k(t)+M))} \\
	&= \normz{D\Lambda(\phi_k(t))\cdot \phi'_k(t)-D\Lambda(0,\dots,0,E^k(t)+M)\cdot (0,\dots,0,(E^k)'(t))^T} \\
	&= \big\lVert D\Lambda(\phi_k(t))\cdot \left(\phi'_k(t)-(0,\dots,0,(E^k)'(t))^T\right)  \\
	&\quad +\left(D\Lambda(\phi_k(t))-D\Lambda(0,\dots,0,E^k(t)+M)\right)\cdot (0,\dots,0,(E^k)'(t))^T \big\rVert_2 \\
	&\leq \norm{D\Lambda(\phi_k(t))}\cdot \normz{\phi'_k(t)-(0,\dots,0,(E^k)'(t))^T} \\
	&\quad +\norm{D\Lambda(\phi_k(t))-D\Lambda(0,\dots,0,E^k(t)+M)}\cdot \normz{(0,\dots,0,(E^k)'(t))} \label{eq:d/dt}
	\end{aligned}
	\end{equation}
	We estimate the first part on the right hand side. Since, by \protect \eqref{eq:phi1}
	\begin{equation}
	\begin{aligned}
	\normz{\phi'_k(t)-(0,\dots,0,(E^k)'(t))^T} &= \normz{\left(0,\dots,0,\left(\frac{E'(E^k(t))}{E^{k+1}(t)+M+a}-1\right)\cdot (E^k)'(t)\right)}\\
	&=\left(1-\frac{E^{k+1}(t)+1}{E^{k+1}(t)+M+a} \right)\cdot (E^k)'(t) \\
	&= \frac{M+a-1}{E^{k+1}(t)+M+a}\cdot (E^k)'(t) \\
	&\leq \frac{a}{E^{k+1}(t)}\cdot (E^k)'(t),
	\end{aligned}
	\end{equation}
	and, by \eqref{eq:inverse} and \eqref{eq:log Ek+1 = Ek+Rk}
	\begin{equation}
	\begin{aligned}
	\norm{D\La(\phi_k(t)} &\leq \dfrac{c_4}{\normz{\phi_k(t)}} \\
	&\leq \dfrac{c_4}{\phi_{k,d}(t)} \\
	&= \dfrac{c_4}{\log\left(E^{k+1}(t)+M+a\right)} \\
	&\leq \dfrac{c_4}{E^k(t)}. \label{eq:d/dt2}
	\end{aligned}
	\end{equation}
	
	Now we estimate the second part on the right hand side. Since the estimates in \protect \eqref{eq:DF(Lx)-DF(Ly)} are symmetric in $x$ and $y$ we may assume without loss of generality that $\normz{x}\leq \normz{y}$.
	
	Using the fact that for all $A,B \in GL(d,\R)$ 
	\begin{equation}
	A^{-1}-B^{-1} = B^{-1}\cdot (B-A)\cdot A^{-1}
	\end{equation}
	and for all $x\in\mathbb{H}_{\geq M}$
	\begin{equation}
	D\La(x) \cdot DF(\La(x)) = I,
	\end{equation}
	we obtain from Lemma \ref{estimate_I_2} for all $x,y\in\mathbb{H}_{\geq M}$ 
	\begin{align}
	&\hspace{0.49cm}\norm{D\La(x)-D\La(y)} \\
	&=\norm{DF(\La(x))^{-1}-DF(\La(y))^{-1}} \\
	&= \norm{DF(\La(y))^{-1}\cdot(DF(\La(y))-DF(\La(x)))\cdot DF(\La(x))^{-1}} \\
	&\leq \norm{DF(\La(y))^{-1}}\cdot \norm{DF(\La(y))-DF(\La(x))}\cdot \norm{DF(\La(x))^{-1}} \\
	&\leq c_7\min\left\{\normz{x}^{1-\beta},\normz{y}^{1-\beta}\right\}  \cdot \norm{D\La(x)} \cdot \norm{D\La(y)} \cdot \max\left\{\normz{x-y},\normz{x-y}^\beta\right\} \\
	&= c_7\normz{x}^{1-\beta} \cdot \norm{D\La(x)} \cdot \norm{D\La(y)} \cdot \max\left\{\normz{x-y},\normz{x-y}^\beta\right\} 
	\end{align}
	With inequality \eqref{eq:inverse} this yields for all $x,y\in\mathbb{H}_{\geq M}$
	\begin{equation}
	\begin{aligned}
	\norm{D\La(x)-D\La(y)} &\leq c_4^2c_7\normz{x}^{1-\beta} \cdot \dfrac{\max\left\{\normz{x-y},\normz{x-y}^\beta\right\}}{\normz{x}\cdot \normz{y}} \\
	&=c_4^2c_7\cdot  \dfrac{\max\left\{\normz{x-y},\normz{x-y}^\beta\right\}}{\normz{x}^\beta\cdot \normz{y}}. 	 \label{eq:A^-1 - B^-1}
	\end{aligned}
	\end{equation}
	Notice that we obtain from \protect \eqref{eq:E^l(t)} for all $k\in\N$ 
	\begin{equation}
	\dfrac{1}{\normz{\phi_k(t)}^\beta\cdot \left(E^k(t)+M\right)} \leq \dfrac{1}{E^k(t)^{1+\beta}}.
	\end{equation}
	Recall that by \protect \eqref{eq:log Ek+1 = Ek+Rk} we have  $\log(E^{k+1}(t)+M+a)\leq E^k(t)+\log(M+a)$ for all $k\in\N$.  With 
	\begin{equation}
	c_8:= c_4^2c_7(d+\log(M+a)+M)
	\end{equation}
	and putting $x:=\phi_k(t)=(v_1+2s_{k,1},\dots,v_{d-1}+2s_{k,d-1},\log(E^{k+1}(t)+M+a))$ and $y:=(0,\dots,0,E^k(t)+M)$ we obtain
	\begin{equation}
	\begin{aligned}
	&\hspace{0.49cm}\norm{D\Lambda(\phi_k(t))-D\Lambda(0,\dots,0,E^k(t)+M)} \\ 
	&\leq c_4^2c_7\cdot \dfrac{\max\left\{\normz{\left(v+2s_k,\log(M+a)+M\right)},\normz{(v+2s_k,\log(M+a)+M)}^\beta\right\}}{E^k(t)^{1+\beta}}  \\
	&\leq c_4^2c_7\cdot \dfrac{\max\left\{d+2\normz{s_k}+\log(M+a)+M,\left(d+2\normz{s_k}+\log(M+a)+M\right)^\beta\right\}}{E^k(t)^{1+\beta}}\\
	&\leq c_4^2c_7 \cdot \dfrac{d+2\normz{s_k}+\log(M+a)+M}{E^k(t)^{1+\beta}}\\
	&\leq c_8 \cdot \dfrac{2\normz{s_k}+1}{E^k(t)^{1+\beta}}
	\end{aligned}
	\end{equation}
	for all $k\in\N$. Together this yields
	\begin{equation}
	\begin{aligned}
	&\hspace{0.49cm}\norm{D\Lambda(\phi_k(t))-D\Lambda(0,\dots,0,E^k(t)+M)}\cdot \normz{(0,\dots,0,(E^k)'(t))} \\
	&\leq c_8 \cdot \dfrac{2\normz{s_k}+1}{E^k(t)^{1+\beta}}\cdot (E^k)'(t) \\
	&=I_{2,k}(t)\cdot (E^k)'(t). \label{eq:d/dt3}
	\end{aligned}
	\end{equation}
	Finally we obtain the conclusion from \protect \eqref{eq:d/dt}, \protect \eqref{eq:d/dt2} and \protect \eqref{eq:d/dt3}.
\end{proof}
From \protect \eqref{eq:g_k'} we know that 
\begin{equation}
\begin{aligned}
g_{k-1}'(t) &= \prod_{l=1}^{k}D\La\left(\left(\Comp_{j=l}^{k-1}L_j\right)\left(0,\dots,0,E^{k}(t)+M\right)\right)\cdot \left(0,\dots,0,\left(E^k\right)'(t)\right)^T \\
&= \prod_{l=1}^{k-1}D\La\left(\left(\Comp_{j=l}^{k-1}L_j\right)\left(0,\dots,0,E^{k}(t)+M\right)\right)\cdot \dfrac{d}{dt}(L_{k-1}(0,\dots,0,E^k(t)+M))
\end{aligned}
\end{equation}
for all $k\in\N$. Moreover we obtain with \eqref{eq:phi_k} and \protect \eqref{eq:phi2} 
\begin{align}
g_k'(t)&=\prod_{l=1}^{k+1}\left(D\La\left(\left(\Comp_{j=l}^{k}L_j\right)\left(0,\dots,0,E^{k+1}(t)+M\right) \right)  \right)\cdot \left(0,\dots,0,\left(E^{k+1}\right)'(t)\right)^T \\
&= \prod_{l=1}^{k-1}\left(D\La\left(\left(\Comp_{j=l}^{k-1}L_j\right)(\phi_k(t)) \right)  \right)\cdot D\La\left(L_k\left(0,\dots,0,E^{k+1}(t)+M\right)\right) \\
&\hspace{1cm}\cdot D\La\left(0,\dots,0,E^{k+1}(t)+M\right) \cdot \left(0,\dots,0,\left(E^{k+1}\right)'(t)\right)^T \\
&= \prod_{l=1}^{k-1}D\La\left(\left(\Comp_{j=l}^{k-1}L_j\right)(\phi_k(t))\right)\cdot D\La\left(\phi_k(t)\right)\cdot \phi_k'(t) \\
&= \prod_{l=1}^{k-1}D\La\left(\left(\Comp_{j=l}^{k-1}L_j\right)(\phi_k(t))\right)\cdot \dfrac{d}{dt}(L_{k-1}\circ \phi_k)(t).
\end{align} 
Putting
\begin{align}
A_k(t) &:= \prod_{l=1}^{k-1}D\La\left(\left(\Comp_{j=l}^{k-1}L_j\right)(\phi_k(t))\right) \\
B_k(t) &:= \prod_{l=1}^{k-1}D\La\left(\left(\Comp_{j=l}^{k-1}L_j\right)(0,\dots,0,E^{k}(t)+M)\right)
\end{align}
for all $k\in\N$ and $t\in [0,\infty)$, we obtain
\begin{equation}
\begin{aligned}
&\hspace{0.49cm}g_k'(t)-g_{k-1}'(t) \\
&= A_k(t)\cdot  \dfrac{d}{dt}(L_{k-1}\circ \phi_k)(t) 
- B_k(t) \cdot \dfrac{d}{dt}(L_{k-1}(0,\dots,0,E^k(t)+M)). \label{eq:g_k'-g_k-1'}
\end{aligned}
\end{equation}
At this point we need suitable estimates for $A_k(t)$ and $B_k(t)$.
\begin{lem} \label{A_k-B_k}
	For all $k\in\N_{\geq 2}$ and $t\in [0,\infty)$ we have
	\begin{equation}
	\begin{aligned}
	&\hspace{0.49cm}\norm{A_k(t)-B_k(t)} \\
	&\leq \sum_{r=1}^{k-1} \left(\prod_{l=1}^{r-1} \norm{X_{l,k-1}(t)}\cdot \norm{X_{r,k-1}(t)-Y_{r,k-1}(t)}\cdot \prod_{s=r+1}^{k-1} \norm{Y_{s,k-1}(t)}\right), \label{eq:A_k-B_k}
	\end{aligned}
	\end{equation}
	where for $l\in\{0,\dots,k-1\}$
	\begin{align}
	X_{l,k-1}(t) &:= D\La\left(\left(\Comp_{j=l}^{k-1}L_j\right)(\phi_k(t))\right), \\
	Y_{l,k-1}(t) &:= D\La\left(\left(\Comp_{j=l}^{k-1}L_j\right)(0,\dots,0,E^{k}(t)+M)\right).
	\end{align}
\end{lem}
\begin{proof}
	By the triangle inequality and the submultiplicity of the operator norm we obtain
	\begin{align}
	\norm{A_k(t)-B_k(t)} &= \norm{\prod_{l=1}^{k-1}X_{l,k-1}(t)-\prod_{l=1}^{k-1}Y_{l,k-1}(t)} \\
	&= \Bigg\lVert\prod_{l=1}^{k-1}X_{l,k-1}(t)-\left(\prod_{l=1}^{k-2}X_{l,k-1}(t)\right)\cdot Y_{k-1,k-1}(t)   \\
	&\hspace{1cm}+ \left(\prod_{l=1}^{k-2}X_{l,k-1}(t)\right)\cdot Y_{k-1,k-1}(t)-\prod_{l=1}^{k-1}Y_{l,k-1}(t) \Bigg\rVert \\
	&\leq  \prod_{l=1}^{k-2}\norm{X_{l,k-1}(t)}\cdot \norm{X_{k-1,k-1}(t)-Y_{k-1,k-1}(t)} \\
	&\hspace{1cm}+ \norm{\prod_{l=1}^{k-2}X_{l,k-1}(t)-\prod_{l=1}^{k-2}Y_{l,k-1}(t)}\cdot \norm{Y_{k-1,k-1}(t)}.
	\end{align}
	Repeating this procedure we obtain our result.
\end{proof}
\begin{proof}[Proof of Theorem \ref{thm1}.] Let $\varepsilon >0$. We will show that there is a constant $C_1>0$ such that
	\begin{equation}
	\normz{g'_k(t)-g'_{k-1}(t)} \leq C_1\alpha^{k-1}
	\end{equation}
	for $t\in[t_{\ul{s}}+\varepsilon,\infty)$ and large $k\in\N$, where $\alpha$ is the constant given in \eqref{eq:alpha}. This implies that $(g_k')_k$ converges locally uniformly on $(t_{\ul{s}},\infty)$ and thus the hairs of $f$ are $C^1$-smooth which yields Theorem \ref{THM}. \\
	First of all we deduce from \protect \eqref{eq:g_k'-g_k-1'} that 
	\begin{equation}
	\begin{aligned}
	&\hspace{0.49cm}\normz{g'_k(t)-g'_{k-1}(t)} \\
	&= \normz{A_k(t)\cdot  \dfrac{d}{dt}(L_{k-1}\circ \phi_k)(t) 
		- B_k(t) \cdot \dfrac{d}{dt}(L_{k-1}(0,\dots,0,E^k(t)+M))} \\
	&= \Bigg\lVert A_k(t)\cdot \dfrac{d}{dt}(L_{k-1}\circ \phi_k)(t)-A_k(t)\cdot \dfrac{d}{dt}(L_{k-1}(0,\dots,0,E^k(t)+M)) \\
	&\hspace{0.5cm} +\left(A_k(t)-B_k(t)\right)\cdot \dfrac{d}{dt}(L_{k-1}(0,\dots,0,E^k(t)+M))\Bigg\rVert_2 \\
	&\leq \norm{A_k(t)}\cdot \normz{\dfrac{d}{dt}(L_{k-1}\circ \phi_k)(t)-\dfrac{d}{dt}(L_{k-1}(0,\dots,0,E^k(t)+M))} \\
	&\hspace{0.5cm}+\norm{A_k(t)-B_k(t)}\cdot \normz{\dfrac{d}{dt}(L_{k-1}(0,\dots,0,E^k(t)+M))}.
	\end{aligned}
	\end{equation}
	With inequality \protect \eqref{eq:phi_k' < E^k'} we obtain 
	\begin{equation}
	\normz{\dfrac{d}{dt}(L_{k-1}(0,\dots,0,E^k(t)+M))} = \normz{\phi_{k-1}'(t)} \leq \left(E^{k-1}\right)'(t).
	\end{equation}
	Thus using Lemma \ref{estimate_1} we obtain
	\begin{equation}
	\begin{aligned}
	&\hspace{0.49cm}\normz{g'_k(t)-g'_{k-1}(t)} \\
	&\leq\norm{A_k(t)}\cdot \big(I_{1,k}(t)+I_{2,k}(t)\big)\cdot \left(E^k\right)'(t)+ \norm{A_k(t)-B_k(t)}\cdot \left(E^{k-1}\right)'(t).
	\end{aligned}
	\end{equation}
	For simplicity we split the last step up into three parts which we will estimate separately. Therefore we define
	\begin{align}
	J_{1,k}(t) &:= \norm{A_k(t)}\cdot I_{1,k}(t)\cdot \left(E^k\right)'(t) \\
	J_{2,k}(t) &:= \norm{A_k(t)}\cdot I_{2,k}(t)\cdot \left(E^k\right)'(t) \\
	J_{3,k}(t) &:= \norm{A_k(t)-B_k(t)}\cdot \left(E^{k-1}\right)'(t).
	\end{align}
	Then the upper inequality reads as
	\begin{equation}
	\normz{g'_k(t)-g'_{k-1}(t)} \leq J_{1,k}(t) + J_{2,k}(t) + J_{3,k}(t).
	\end{equation}
	\textbf{1. Estimate of $J_{1,k}(t)$.} 
	
	Using inequality \eqref{eq:deriv<alpha} we obtain
	\begin{equation}
	J_{1,k}(t) \leq \frac{ac_4\alpha^{k-1}}{E^{k+1}(t)\cdot E^k(t)}\cdot \left(E^k\right)'(t).
	\end{equation}
	Since $t \in [t_{\ul{s}}+\varepsilon,\infty) \subset (0,\infty)$ this implies
	\begin{equation}
	\lim\limits_{k\to\infty}	\frac{\left(E^k\right)'(t)}{E^{k+1}(t)} =  \lim\limits_{k\to\infty}	\frac{\prod_{j=1}^{k} \left(E^j(t)+1\right)}{E^{k+1}(t)} =0
	\end{equation}
	and hence
	\begin{equation}
	\lim\limits_{k\to\infty}	\frac{\left(E^k\right)'(t)}{E^{k+1}(t)\cdot E^{k}(t)} = 0.
	\end{equation}
	So there exists $c_9>0$ such that
	\begin{equation}
	J_{1,k}(t)\leq c_9\alpha^{k-1}.
	\end{equation}
	\textbf{2. Estimate of $J_{2,k}(t)$.} 
	
	We obtain by the submultiplicity of the operator norm, Lemma \ref{Rschlange} and inequality \eqref{eq:inverse} 
	\begin{align}
	\norm{A_k(t)}& \leq \prod_{l=1}^{k-1}\norm{D\La\left(\left(\Comp_{j=l}^{k-1}L_j\right)(\phi_k(t))\right)} 
	\leq c_4^{k-1}\cdot \left(\prod_{l=1}^{k-1}E^l(t)\right)^{-1}.
	\end{align}
	With \protect \eqref{eq:I_2,k} we have
	\begin{align}
	J_{2,k}(t) &\leq c_4^{k-1}\cdot \dfrac{\left(E^k\right)'(t)}{\prod_{l=1}^{k-1}E^l(t)}\cdot I_{2,k}(t)   \\
	&\leq c_4^{k-1}c_8\dfrac{2\normz{s_k}+1}{E^k(t)^\beta}\cdot \prod_{l=1}^{k}\left(1+\dfrac{1}{E^l(t)}\right) \\
	&= c_8\dfrac{c_4^{k-1}}{E^k(t)^{\beta /2}}\cdot \dfrac{E^k(t_k)+1}{E^k(t)^{\beta /2}} \cdot \prod_{l=1}^{k}\left(1+\dfrac{1}{E^l(t)}\right).
	\end{align}
	At this point notice that it follows from  \eqref{eq:E(t)} and \eqref{eq:t_s}, since $t>t_{\ul{s}}$, that
	\begin{equation}
	\lim\limits_{k\to\infty} \dfrac{E^k(t_k)}{E^k(t)^{\beta /2}} =0
	\end{equation}
	and
	\begin{equation}
	\prod_{l=1}^{\infty}\left(1+\dfrac{1}{E^l(t)}\right)<\infty \label{eq:prod}
	\end{equation}
	for $t>t_{\ul{s}}$. Moreover
	\begin{equation}
	\dfrac{c_4^{k-1}}{E^k(t)^{\beta /2}} \leq \alpha^{k-1}
	\end{equation}
	for large $k\in\N$. Since $t\in [t_{\ul{s}}+\varepsilon,\infty)$ there is a constant $c_{10}>0$ such that we have for large $k \in\N$ 
	\begin{equation}
	J_{2,k}(t) \leq c_{10}\alpha^{k-1}.
	\end{equation}
	\textbf{3. Estimate of $J_{3,k}(t)$.} 
	
	We want to estimate the right hand side of \protect \eqref{eq:A_k-B_k} step by step. From \eqref{eq:inverse} and \protect \eqref{eq:E^l(t)} we obtain for all $r\in\{1,\dots,k-1\}$
	\begin{equation}
	\prod_{l=1}^{r-1} \norm{X_{l,k-1}(t)} \leq c_4^{r-1} \cdot \left(\prod_{l=1}^{r-1}\normz{\left(\Comp_{j=l}^{k-1}L_j\right)(\phi_k(t))}\right)^{-1} \leq c_4^{r-1}\cdot \left(\prod_{l=1}^{r-1}E^l(t)\right)^{-1}.
	\end{equation}
	Similarly Remark \ref{doubleschlange} yields 
	\begin{equation}
	\prod_{l=r+1}^{k-1} \norm{Y_{l,k-1}(t)}  \leq c_4^{k-r-1}\cdot \left(\prod_{l=r+1}^{k-1}E^l(t)\right)^{-1}.
	\end{equation}
	Together this implies that
	\begin{equation}
	\prod_{l=1}^{r-1} \norm{X_{l,k-1}(t)} \cdot \prod_{l=r+1}^{k-1} \norm{Y_{l,k-1}(t)} \leq  c_4^{k-2}\cdot \left(\prod_{\substack{l=1 \\ l\neq r}}^{k-1}E^l(t)\right)^{-1}.
	\end{equation}
	Recall that for 
	\begin{equation}
	x = \left(\Comp_{j=r}^{k-1}L_j\right)(\phi_k(t)) \qquad \text{and} \qquad y= \left(\Comp_{j=r}^{k-1}L_j\right)\left(0,\dots,0,E^{k}(t)+M\right)
	\end{equation}
	the inequalities \eqref{eq:E^l(t)} and \protect\eqref{eq:A^-1 - B^-1} ensure that
	\begin{equation}
	\norm{D\La(x)-D\La(y)} \leq c_4^2c_7 \dfrac{\max\left\{\normz{x-y},\normz{x-y}^\beta\right\}}{E^r(t)^{1+\beta}}.
	\end{equation}
	Notice that there is a constant $c_{11}\geq 1$ such that for large $k\in\N$
	\begin{equation}
	\normz{L_{k-1}(\phi_k(t))-L_{k-1}\big(0,\dots,0,E^{k}(t)+M\big)}\leq c_{11}.
	\end{equation}
	For $r=k-1$ we obtain with \eqref{eq:E^l(t)} and \eqref{eq:A^-1 - B^-1} 
	\begin{equation}
	\begin{aligned}
	&\hspace{0.49cm}\norm{X_{k-1,k-1}(t)-Y_{k-1,k-1}(t)} \\
	&= \norm{D\La\left(L_{k-1}(\phi_k(t))\right)-D\La\left(L_{k-1}\big(0,\dots,0,E^{k}(t)+M\big)\right)} \\
	&\leq \dfrac{c_4^2c_7c_{11}}{E^{k-1}(t)^{1+\beta}}.
	\end{aligned}
	\end{equation}
	Thus
	\begin{equation}
	\begin{aligned}
	&\hspace{0.49cm}\prod_{l=1}^{k-2} \norm{X_{l,k-1}(t)} \cdot \norm{X_{k-1,k-1}(t)-Y_{k-1,k-1}(t)} \cdot \prod_{l=k}^{k-1} \norm{Y_{l,k-1}(t)}\cdot \left(E^{k-1}\right)'(t) \\
	&\leq c_7c_{11}\dfrac{c_4^k}{E^{k-1}(t)^\beta}\cdot \left(\prod_{l=1}^{k-1}E^l(t)\right)^{-1}\cdot \left(E^{k-1}\right)'(t) \\
	&= c_7c_{11}\dfrac{c_4^k}{E^{k-1}(t)^\beta}\cdot \prod_{l=1}^{k-1}\left(1+\dfrac{1}{E^l(t)}\right).
	\end{aligned}
	\end{equation}
	Since
	\begin{equation}
	\dfrac{c_4^k}{E^{k-1}(t)^\beta} \leq \dfrac{\alpha^{k-1}}{k}
	\end{equation}
	for large $k\in\N$ and $t>t_{\ul{s}}$, we obtain for $r=k-1$ from \protect \eqref{eq:prod} 
	\begin{equation}
	\begin{aligned}
	&\hspace{0.49cm}\prod_{l=1}^{k-2} \norm{X_{l,k-1}(t)}\cdot \norm{X_{k-1,k-1}(t)-Y_{k-1,k-1}(t)}\cdot \left(E^{k-1}\right)'(t) \\
	&\leq \frac{c_{12}}{k}\alpha^{k-1} \label{eq:c12/k1}
	\end{aligned}
	\end{equation}
	for a constant $c_{12}>0$ and large $k\in\N$.
	
	For $r\neq k-1$ notice that with \protect \eqref{eq:L(x)-L(y)} and \eqref{eq:c_4/min} we obtain
	\begin{equation}
	\begin{aligned}
	&\hspace{0.49cm}\displaystyle\normz{\left(\Comp_{j=r}^{k-1}L_j\right)(\phi_k(t))-\left(\Comp_{j=r}^{k-1}L_j\right)\big(0,\dots,0,E^{k}(t)+M\big)} \\ 
	&\leq \alpha^{k-r-2} \normz{L_{k-2}(L_{k-1}(\phi_k(t)))-L_{k-2}(L_{k-1}\big(0,\dots,0,E^k(t)+M)\big)} \\
	&\leq \alpha^{k-r-2}c_4\pi \dfrac{\normz{L_{k-1}(\phi_k(t))-L_{k-1}(0,\dots,0,E^k(t)+M)}}{\min\left\{ \normz{L_{k-1}(\phi_k(t))}, \normz{L_{k-1}(0,\dots,0,E^k(t)+M)}\right\}} \\
	&\leq \dfrac{c_4c_{11}\pi }{E^{k-1}(t)}.
	\end{aligned}
	\end{equation}
	Since $E^{k-1}(t)\geq E^{k-1}(t)^\beta$ for all $t\geq 0$ and large $k\in\N$, this implies
	\begin{equation}
	\begin{aligned}
	&\hspace{0.49cm}\prod_{l=1}^{r-1} \norm{X_{l,k-1}(t)}\cdot \norm{X_{r,k-1}(t)-Y_{r,k-1}(t)}\cdot \prod_{s=r+1}^{k-1} \norm{Y_{s,k-1}(t)}\cdot \left(E^{k-1}\right)'(t) \\
	&\leq c_4^{2}c_7\cdot\dfrac{1}{E^r(t)^{1+\beta}}\cdot\dfrac{c_4c_{11}\pi}{E^{k-1}(t)^\beta}\cdot c_4^{k-2} \cdot\left(\prod_{\substack{l=1 \\ l\neq r}}^{k-1}E^l(t)\right)^{-1}\cdot \left(E^{k-1}\right)'(t) \\
	&=\dfrac{c_7c_{11}\pi}{E^r(t)^{\beta}} \cdot \dfrac{c_4^{k+1}}{E^{k-1}(t)^{\beta}} \cdot \prod_{l=1}^{k-1}\left(1+\dfrac{1}{E^l(t)}\right) \\
	&\leq \dfrac{c_{12}}{k}\alpha^{k-1}. \label{eq:c12/k2}
	\end{aligned}
	\end{equation}
	Therefore we get again the estimate from above for large $k\in\N$ which finally yields with \protect \eqref{eq:c12/k1} and \protect \eqref{eq:c12/k2}
	\begin{equation}
	\begin{aligned}
	J_{3,k}(t) &= \norm{A_k(t)-B_k(t)}\cdot \left(E^{k-1}\right)'(t) \\
	&\leq\sum_{r=1}^{k-1} \frac{c_{12}}{k}\alpha^{k-1} \\
	&\leq c_{12}\alpha^{k-1}
	\end{aligned}
	\end{equation}
	for large $k \in\N$. Altogether we obtain with $c_{13}:=c_9+c_{10}+c_{12}$
	\begin{align}
	\normz{g'_k(t)-g'_{k-1}(t)} &\leq J_{1,k}(t) + J_{2,k}(t) + J_{3,k}(t) 
	\leq c_{13}\alpha^{k-1}. \label{eq:g_k'-g_{k-1}'}
	\end{align}
	Since for all $k\in\N$
	\begin{equation}
	g'_k(t) = g_0'(t) + \sum_{l=1}^{k}(g_l'(t)-g_{l-1}'(t))
	\end{equation}
	we obtain from \protect \eqref{eq:g_k'-g_{k-1}'} the uniform convergence of $(g_k')$ on 
	$[t_{\ul{s}}+\varepsilon,\infty)$. This yields for $\varepsilon\to 0$ the locally uniform convergence of $(g'_k)$ on $(t_{\ul{s}},\infty)$ which proves Theorem \ref{thm1}.
\end{proof}
\begin{remark}
	It seems plausible that this result generalizes to the case of higher derivatives using similar methods, but in a more technical way.
\end{remark}
\section{Remarks} \label{Remarks}
	In Theorem \ref{THM} we assumed that the function $h$ is $C^1$ in the interior of $Q$ and $Dh$ is H\"older continuous there. Karpi\'nska \cite{Ka99} proved her result using the fact that points in $J\setminus C$ escape to $\infty$ in a comparatively small and parabolic-like domain. To put it in a precise form in our context, Bergweiler \cite{Ber10} considered the function
	\begin{equation}
	\psi \colon [1,\infty) \to [1,\infty), \, \psi(x)=\exp\left(\sqrt{\log(x)}\right).
	\end{equation} 
	Then for all $\varepsilon>0$ we obtain
	\begin{equation}
	\lim\limits_{x\to\infty} \frac{\psi(x)}{x^\varepsilon} = \lim\limits_{x\to\infty} \exp\left(-\varepsilon \log(x)+\sqrt{\log(x)}\right) = 0. 
	\end{equation}
	The parabolic-like domain from above than has the form
	\begin{equation}
	\Omega = \{x\in\R^d : x_d>M \text{ and } \normz{\tilde{x}}< \psi(x_d)^2\},
	\end{equation}
	where $x=(\tilde{x},x_d)$.
	In \cite{Ber10} it was shown that if $x\in J\setminus C$, then $f^k(x) \in \Omega$ for large $k\in\N$. Using this information it is enough to require that $h$ satisfies the conditions of Theorem \ref{THM} on a neighborhood of the point $v$ from \protect \eqref{eq:v} to obtain differentiability of the `tails` of the hairs \cite[\S3]{SZ03}. Moreover it can be shown that if $h$ is $C^1$ on a suitable compact subset of $\operatorname{int}(Q)$, we obtain the differentiability of the hairs except for the endpoints. This shows that one can relax the condition on $Dh$ in such a way that $Dh$ needs to be only locally H\"older continuous.

\end{document}